\documentclass[twoside,11pt,fleqn]{article}
\usepackage{latexsym}
\usepackage{amsmath, amsthm}
\usepackage{graphicx}
\usepackage{amsfonts}
\usepackage{amssymb}

\DeclareMathOperator{\negg}{\mathrm{neg}}
\DeclareMathOperator{\oneg}{\mathrm{oneg}}

\DeclareMathOperator{\maj}{\mathrm{maj}}
\DeclareMathOperator{\fmaj}{\mathrm{fmaj}}
\DeclareMathOperator{\Dmaj}{\mathrm{Dmaj}}
\DeclareMathOperator{\omaj}{\mathrm{omaj}}
\DeclareMathOperator{\emaj}{\mathrm{emaj}}
\DeclareMathOperator{\odes}{\mathrm{odes}}
\DeclareMathOperator{\edes}{\mathrm{edes}}
\DeclareMathOperator{\des}{des}

\DeclareMathOperator{\ofmaj}{\mathrm{ofmaj}}
\DeclareMathOperator{\efmaj}{\mathrm{efmaj}}
\DeclareMathOperator{\eneg}{\mathrm{eneg}}
\DeclareMathOperator{\oDmaj}{\mathrm{oDmaj}}
\DeclareMathOperator{\eDmaj}{\mathrm{eDmaj}}
\DeclareMathOperator{\Neg}{\mathrm{Neg}}
\DeclareMathOperator{\modue}{(\mathrm{mod}~2)}

\theoremstyle{plain}
\newtheorem{thm}{Theorem}[section]
\newtheorem{pro}[thm]{Proposition}
\newtheorem{lem}[thm]{Lemma}

\newtheorem{cor}[thm]{Corollary}

\theoremstyle{definition}

\theoremstyle{remark}

\newcommand{\sm}{\sigma^{-1}}
\newcommand{\N}{\mathbb N}
\newcommand{\Z}{\mathbb Z}
\newcommand{\PP}{\mathbb P}

\newcommand{\eqdef}{:=}

\begin{document}

\begin{center}

{\Large \bf Odd and even major indices and one-dimensional characters for classical Weyl groups
\footnote{2010 Mathematics Subject
Classification: Primary 05A15; Secondary 05E15, 20F55.}}
 \vspace{0.8cm}

 {Francesco Brenti} \\

Dipartimento di Matematica \\

Universit\'{a} di Roma ``Tor Vergata''\\

Via della Ricerca Scientifica, 1 \\

00133 Roma, Italy \\

{\em brenti@mat.uniroma2.it } \\

 \vspace{0.5cm}

 Paolo Sentinelli  \\

%
%
%

{\em paolosentinelli@gmail.com } \\

\end{center}

\vspace{1cm}

\begin{abstract}
We define and study odd and even analogues of the major index statistics for the
classical Weyl groups. More precisely, we show that the generating functions
of these statistics, twisted by the one-dimensional characters of the corresponding
groups, always factor in an explicit way. In particular, we obtain odd and
even analogues of Carlitz's identity, of the Gessel-Simion Theorem, and a
parabolic extension, and refinement, of a result of Wachs.
\end{abstract}

\section{Introduction}
In recent years a new statistic on the symmetric group has been introduced and studied in relation
with vector spaces over finite fields equipped with a certain
quadratic form
(\cite{KV}). This statistic combines combinatorial and parity conditions and is now known as
the odd inversion number, or odd length (\cite{BC}, \cite{BC3}).
Analogous statistics have later been defined and studied for the hyperoctahedral and even
hyperoctahedral groups
(\cite{SV1}, \cite{SV2}, \cite{BC2}), and more recently for all Weyl groups (\cite{BC3}).
A crucial property of this new statistic is that its signed (by length) generating function over
the corresponding Weyl group always factors explicitly (\cite{BC3},\cite{Ste}).

Another line of research in the last 20 years has been the definition and study of analogues
of the major index statistic for the other classical Weyl groups, namely for the hyperoctahedral
and even hyperoctahedral groups (see, e.g., \cite{AR}, \cite{BiagAAM}, \cite{Biagioli-Caselli}, \cite{CF94}, \cite{CF95a}, \cite{CF95b}, \cite{Rei93a},
\cite{Rei93b}, \cite{Ste94}) and for finite Coxeter groups (\cite{Sentinelli}).
It is now generally recognized
that, among these, the ones with the best properties are those first defined by Adin and Roichman in \cite{AR} for the hyperoctahedral group and by Biagioli and Caselli in \cite{Biagioli-Caselli}
for the even hyperoctahedral group.

Our purpose in this work is to define odd (and even) analogues of these major index
statistics for the classical Weyl groups and show that their generating
function twisted by the one-dimensional characters of the corresponding Weyl group always factors in an explicit way.
More precisely, we show that certain multivariate refinements of these generating functions always factor explicitly.
As consequences of our results we obtain  odd and even analogues of Carlitz's identity \cite{Carlitz},
which involves overpartitions, of the Gessel-Simion Theorem
(see, e.g., \cite[Theorem 1.3]{AGR}), of several other results appearing in the literature (\cite[Theorems 5.1, 6.1, 6.2]{AGR} and \cite[Theorem 4.8]{Biagioli-2006}) and an extension, and refinement of a result of Wachs (\cite{Wachs}).

The organization of the paper is as follows. In the next section we recall some definitions
and results that are used in the sequel. In \S 3 we define and study odd and even analogues of the
major index and descent statistics of the symmetric group. In particular, we obtain odd and even
analogues of Carlitz's identity, of the Gessel-Simion Theorem, and a parabolic extension, and refinement, of a result of Wachs.
In \S 4 we define odd and even analogues of the major index statistics introduced in
\cite{AR} and \cite{Biagioli-Caselli} for the classical Weyl groups of types $B$ and $D$, respectively, and of the usual
descent statistics on these groups. More precisely, we compute a multivariate refinement of the generating
functions of these statistics twisted by the one-dimensional characters of the corresponding groups and show
that they always factor explicitly. Finally, in \S 5, we show that, under some mild
and natural hypotheses, there is no ``odd major index'' that is equidistributed with the odd
length in the symmetric or hyperoctahedral groups.

\section{Preliminaries}
In this section we recall some notation, definitions, and results that are
used in the sequel.
As $\mathbb{N}$ we denote the set of
non-negative integers and as $\mathbb{P}$ the set of positive integers. If $n\in \mathbb{N}$,
then
$[n]:=\{1,2,...,n\}$ and $[\pm n]:=\{-n,..., -1,1,,...,n\}$, in particular
$[0]=\varnothing$.
For $n\in \mathbb{P}$, in the polynomial ring $\mathbb{Z}[q]$
the $q$-analogue of $n$ is defined by $[n]_q:=\sum \limits_{i=0}^{n-1}q^i$ and the
$q$-factorial by
$[n]_q!:=\prod \limits_{i=1}^n[i]_q$. We also find it convenient to let
$ p_n := (1+(-1)^n)/2$.
The cardinality of a set $X$ is denoted by
$|X|$ and the power set of $X$ by $\mathcal{P}(X)$. For $n,k \in \N$ we let
$\binom{[n]}{k}:=\{A\in \mathcal{P}([n]):|A|=k\}$.

For $n\in \mathbb{N}$, $i \in \mathbb{Z}$, $q \in \mathbb{Q}$ and $J \subseteq [n]$ we let
$J_e := \{j\in J:j\equiv 0\modue\}$, $J_o := \{j\in J:j\equiv 1\modue\}$,
  $J+i := \{i+j:j\in J\} \cap [n]$ and
  $qJ := \{qj : j \in J\}$.

Next we recall some basic results in the theory of
Coxeter groups which
are useful in the sequel. The reader can consult  \cite{BB}
or \cite{Hum} for further details.
Let $(W,S)$ be a Coxeter system. The length of an element $z\in W$ with respect
to $S$ is denoted as
$\ell(z)$. If $J\subseteq S$ and $w \in W$ we let
$W^J:=\{w\in W:\ell(ws)>\ell(w)~\forall~s\in J\}$,
$D(w):=\{s\in S:\ell(ws)<\ell(w)\}$
and, more generally, for any $A\subseteq W$ we let $A^J:=A\cap W^J$.
When the group $W$ is finite, there exists a unique
element $w_0$ of maximal length.

For any $n\in \mathbb{P}$ let $S_n$ be the group of bijections of the set $[n]$.
For $\sigma, \tau \in S_n$ we let $\sigma \tau := \sigma \circ \tau$
(composition of functions). It is well known (see e.g. \cite{BB}) that
this is a Coxeter group with set of generators $\{s_1,s_2,...,s_{n-1}\}$, $s_i$
being, in one line notation,
$12...(i+1)i...n$. Given a permutation
$\sigma = \sigma(1)\sigma(2)...\sigma(n)\in S_n$, the action of
$s_i$ on the right is
given by $\sigma s_i = \sigma(1)\sigma(2)...\sigma(i+1)\sigma(i)...\sigma(n)$, for all $i\in [n-1]$.
As a Coxeter group, identifying $\{s_1,s_2,...,s_{n-1}\}$ with $[n-1]$, we have the following
well known result (see e.g. \cite{BB}).
\begin{pro}
\label{combA}
Let\/ $\sigma \in S_{n}$. Then
$
\ell(\sigma)=|\{(i,j) \in [n]^2 : i<j,\sigma(i)>\sigma(j)\}|,
$
and
$
D(\sigma)=\{i\in [n-1]:\sigma(i)>\sigma(i+1)\}.
$
\end{pro}
\noindent So, given $J \subseteq [n-1]$, $S^J_n=\{\sigma \in S_n:\sigma(i)<\sigma(i+1)~\forall ~i \in J\}$.

For $i \in [n]$ and $A \subseteq S_n$ define $A(i):=\{\sigma \in A: \sm(n)=i\}$. For $A \subseteq \Z$, $A = \{ a_1, \ldots , a_k \}_<$, and $\sigma \in S(A)$ we let
$\tau$ be the only element of $S_{|A|}$ defined by $\sigma(a_i) = a_{\tau(i)}$ for all
$i \in [k]$. We call $\tau$ the {\em flattening} of $\sigma$ and write $F(\sigma)=\tau$.
 Moreover, we define:
 $$i^*:=\left\{
     \begin{array}{ll}
       i-sgn(i), & \hbox{if $i\equiv 0 \modue $;} \\
       i+sgn(i), & \hbox{if $i\equiv 1 \modue $ and $i+sgn(i) \in [\pm n]$;}\\
       i, & \hbox{otherwise,}
     \end{array}
   \right.
 $$ for all $i\in [\pm n]$.

 The elements of $S_n^B$ are the bijective functions $\sigma : [\pm n] \rightarrow [\pm n]$ satisfying
 $-\sigma(i)=\sigma(-i)$, for all $i\in [n]$. We use the window notation. So, for example, the element
 $[-2,1]\in S^B_2$ represents the function $\sigma : [\pm 2] \rightarrow [\pm 2]$ such that
 $\sigma(1)=-2=-\sigma(-1)$ and $\sigma(2)=1=-\sigma(-2)$.
We let $\Neg(\sigma):=\{i\in [n]:\sigma(i)<0\}$, $\negg(\sigma)=|\Neg(\sigma)|$, $s^B_{j} \eqdef (j,j+1)(-j,-j-1)$ for $j=1,...,
n-1$, $s_{0} \eqdef (1,-1)$, and $S_{B} \eqdef \{ s_{0},s^B_1,...,s^B_{n-1} \}$.
It is well
known that $(S_{n}^{B},S_{B})$ is a Coxeter system of type $B_{n}$ and that, identifying $S_B$ with
$[0,n-1]$, the
following holds (see, e.g., \cite[\S 8.1]{BB}).
Given $\sigma \in S^{B}_{n}$ we let
$$
\ell_A(\sigma):=|\{(i,j)\in [n]^2:i<j,\sigma(i)>\sigma(j)\}|.
$$
\begin{pro}
\label{CombB}
Let $\sigma \in {S_{n}^{B}}$. Then
$\ell_B(\sigma)=\ell_A(\sigma)-\sum\limits_{i\in \Neg(\sigma)}\sigma(i)$, and
$D(\sigma):=\{i\in [0,n-1]:\sigma(i)>\sigma(i+1)\}$,
where $\sigma(0):=0$.
\end{pro}

We let $S^D_{n}$ be the subgroup of $S^B_{n}$ defined by
$
S^D_{n} \eqdef \{ \sigma \in S^B_{n} : \, |\Neg(\sigma)| \equiv 0 \pmod{2} \},
$
$\tilde{s}_{0} \eqdef (1,-2)(2,-1)$, and $S_D \eqdef \{ \tilde{s}_{0}
,s^B_{1},...,s^B_{n-1} \}$.
It is then well known that
$(S_{n}^{D},S_D)$ is a Coxeter system of type $D_{n}$, and that the
following holds (see, e.g., \cite[\S 8.2]{BB}).

\begin{pro}
\label{CombD}
Let $\sigma \in {S_{n}^{D}}$. Then
$\ell_D(\sigma)=\ell_B(\sigma)-\negg(\sigma)$, and
$ D(\sigma)= \{ i \in [0,n-1] : \; \sigma(i)>\sigma(i+1) \} $ ,
where $\sigma(0) \eqdef -\sigma(2)$.
\end{pro}
\noindent
For simplicity we often write $B_n$ and $D_n$ respectively in place of $S^B_n$ and $S^D_n$. We refer to \cite[Chapter 8]{BB} for further details on the combinatorics of the groups $S^B_n$ and $S^D_n$.

The \emph{descent number} and the \emph{major index} are the functions $\des : S_n \rightarrow \mathbb{N}$ and
$\maj : S_n \rightarrow \mathbb{N}$ defined respectively by $
\des(\sigma)\eqdef |D(\sigma)|$, and $\maj (\sigma) \eqdef \sum_{i \in D(\sigma)} i$, for all $\sigma \in S_n$.
More generally we let $\des(a)=|D(a)|$ and $\maj (a) \eqdef \sum_{i \in D(a)} i$ for any sequence
$a=(a_1,...,a_n) \in \mathbb{Z}^n$, where $D(a)=\{i\in [n-1]: a_i>a_{i+1}\}$.


Following \cite{AR} and \cite{Biagioli-Caselli} respectively we define the \emph{flag-major index} of an element $\sigma \in S^B_n$ by
$$
\fmaj(\sigma) := 2 \maj(\sigma) + \negg(\sigma),
$$ and the \emph{D-major index}
$$ \Dmaj(\sigma) :=  \fmaj(|\sigma|_n),$$
where $|\sigma|_n := [\sigma(1), \ldots , \sigma(n-1), |\sigma(n)|]$.

%

Recall that a one-dimensional character of a group $G$ is a homomorphism $\chi : G \rightarrow \mathbb{C}\setminus\{0\}$. The
one-dimensional characters of $S_n$ are well known to be the trivial and the alternating one, given by $\sigma \mapsto 1$
and $\sigma \mapsto (-1)^{\ell(\sigma)}$ respectively, for all $\sigma \in S_n$. For the group $S^B_n$ we have the following result (see \cite[Proposition 3.1]{Reiner}):
\begin{pro}
  The hyperoctahedral group $S_n^B$ has four one-dimensional characters, namely $\sigma \mapsto 1$, $\sigma \mapsto (-1)^{\ell(\sigma)}$,
  $\sigma \mapsto (-1)^{\negg(\sigma)}$ and $\sigma \mapsto (-1)^{\ell(\sigma)+\negg(\sigma)}$, for all $\sigma \in S_n^B$.
\end{pro}
The group $S^D_n$ has only the trivial and the alternating one-dimensional characters $\sigma \mapsto 1$
and $\sigma \mapsto (-1)^{\ell(\sigma)}$, for all $\sigma \in S^D_n$ (\cite[Proposition 4.1]{Reiner}).

\section{Type $A$}

In this section we introduce and study odd and even analogues of the descent and major index statistics
for the symmetric group. In particular, we obtain odd and even analogues of Carlitz's identity, of the Gessel-Simion Theorem,
and a parabolic extension, and refinement, of a result of Wachs.

We define functions $\odes,\edes,\omaj,\emaj : S_n \rightarrow \N$ by letting
\[
\odes(\sigma)\eqdef |D(\sigma)_o|, \;\;\; \edes(\sigma)\eqdef |D(\sigma)_e|
\] and
\[
\omaj (\sigma) \eqdef \sum_{i \in D(\sigma)_o} \frac{i+1}{2}, \;\;\;
\emaj (\sigma) \eqdef \sum_{i \in D(\sigma)_e} \frac{i}{2},
\]
for all $\sigma \in S_n$. We call these functions \emph{odd descent number}, \emph{even descent number},
\emph{odd major index}, and \emph{even major index} respectively. So for example,
if $\sigma=81725634$ then $\odes(\sigma)=2$,  $\edes(\sigma)=1$,
$\omaj(\sigma)=3$, and $\emaj(\sigma)=3$.

For any $J\subseteq [n-1]$ we let the \emph{parabolic} $q$-\emph{Eulerian polynomials}
and \emph{parabolic} \emph{signed $q$-Eulerian polynomials}, respectively, be
\[
A^J_n(q,x) := \sum \limits_{\sigma \in S^J_n} q^{\maj(\sigma)} x^{\des(\sigma)}, \\
\]
and
\[
B^J_n(q,x) := \sum \limits_{\sigma \in S^J_n} (-1)^{\ell(\sigma)}q^{\maj(\sigma)} x^{\des(\sigma)}.
\]
So $A^{\varnothing}_n(q,x)=A_n(q,x)$ and $B^{\varnothing}_n(q,x)=B_n(q,x)$ where $A_n(q,x)$ and
$B_n(q,x)$ are, respectively, the $q$-Eulerian polynomials and the signed $q$-Eulerian polynomials, as defined in \cite{Wachs}.

Our goal is to compute the generating functions of odes, omaj, and of edes, emaj, twisted by the one-dimensional
characters of the symmetric groups. Our first result is the odd and even analogue (for $y=-1$ and $x=1$) of the
Gessel-Simion Theorem (see, e.g. \cite[Theorem 1.3]{AGR}) .
\begin{thm}
\label{odd eulerian}
Let $n\in \mathbb{P}$. Then
$$
  \sum \limits_{\sigma \in S_n} y^{\ell(\sigma)} q^{\omaj(\sigma)} x^{\odes(\sigma)}
  =[n]_y !\prod\limits_{i=1}^{\left\lfloor\frac{n}{2}\right\rfloor}\frac{(1+yxq^{i})}{(1+y)},
$$
and
$$
\sum \limits_{\sigma \in S_n} y^{\ell(\sigma)} q^{\emaj(\sigma)} x^{\edes(\sigma)}=
[n]_y ! \prod\limits_{i=1}^{\left\lfloor\frac{n-1}{2}\right\rfloor}\frac{(1+yxq^{i})}{(1+y)}.
$$
\end{thm}
\begin{proof}
  Let, for brevity,
$A^o_n(y,q,x) := \sum \limits_{\sigma \in S_n} y^{\ell(\sigma)} q^{\omaj(\sigma)} x^{\odes(\sigma)}$.
We prove the first equation by induction on $n\geqslant 1$. We have that $A^o_1(y,q,x)=1$, and $A^o_2(y,q,x)=1+yxq$.
So let $n \geq 3$. For $n\equiv 1 \pmod{2}$ we find, by our inductive hypothesis,
\begin{eqnarray*}
 A^o_n(y,q,x)  &=&
\sum \limits_{i\in [n]} \sum \limits_{\substack{u\in S_n\\u(n)=i}} y^{\ell(\sigma)} q^{\omaj(u)}x^{\odes(u)} = \sum \limits_{i\in [n]} y^{n-i} A^o_{n-1}(y,q,x)\\
&=&[n]_y A^o_{n-1}(y,q,x)=[n]_y !\prod\limits_{i=1}^{\left\lfloor\frac{n-1}{2}\right\rfloor}\frac{(1+yxq^{i})}{(1+y)},
\end{eqnarray*}
as desired.
For $n\equiv 0 \pmod{2}$ we have, by our inductive hypothesis,
\begin{eqnarray*}
 A^o_n(y,q,x)
 &=& \sum_{1 \leq i<j \leq n}
 \left( \sum \limits_{\substack{\sigma \in S_n\\\sigma(n-1)=i,\, \sigma(n)=j}} y^{\ell(\sigma)} q^{\omaj(\sigma)}x^{\odes(\sigma)} + \sum \limits_{\substack{\sigma \in S_n\\\sigma(n-1)=j,\, \sigma(n)=i}} y^{\ell(\sigma)} q^{\omaj(\sigma)}x^{\odes(\sigma)} \right) \\
   &=& \sum_{1 \leq i<j \leq n}
  \sum \limits_{\tau \in S_{n-2}} y^{2n-j-i-1+\ell(\tau)} q^{\omaj(\tau)}x^{\odes(\tau)}  \\
 & & + \sum_{1 \leq i<j \leq n} \sum \limits_{\tau \in S_{n-2}} y^{2n-j-i+\ell(\tau)} q^{\omaj(\tau)+\frac{n}{2}}x^{\odes(\tau)+1} \\
   &=& \sum_{1 \leq i<j \leq n} (y^{2n-i-j-1}+ y^{2n-i-j} q^{\frac{n}{2}}x) A^o_{n-2}(y,q,x) \\
   &=& A^o_{n-2}(y,q,x) (1+yq^{\frac{n}{2}}x) \sum_{1 \leq i<j \leq n} y^{2n-i-j-1} \\
   &=& A^o_{n-2}(y,q,x) (1+yq^{\frac{n}{2}}x) \sum_{i=1}^{n-1} y^{n-i-1} [n-i]_y \\
   &=& A^o_{n-2}(y,q,x) (1+yq^{\frac{n}{2}}x) \frac{[n]_y [n-1]_y}{[2]_y},
\end{eqnarray*}
as desired.

The proof for the even statistics is analogous, and is therefore omitted.
\end{proof}

As a corollary of Theorem \ref{odd eulerian} we obtain the odd-even analogue of Carlitz's identity \cite{Carlitz}.
Recall that an overpartition is a partition  where the last occurrence of any number may be overlined (we refer the reader to \cite{Corteel Lovejoy}).
So for example $(1,1,1)$, $(1,1,\overline{1})$, $(2,1)$, $(\overline{2},1)$, $(2,\overline{1})$ and $(\overline{2},\overline{1})$ are the overpartitions of $3$.  
We denote by $\overline{{\cal P}}$ the set of overpartitions.

\begin{cor} \label{corollario over}
  Let $n\in \mathbb{P}$. Then
$$
 \frac{ \sum \limits_{\sigma \in S_n} q^{\omaj(\sigma)} x^{\odes(\sigma)}}{\prod\limits_{i=1}^{\left\lfloor\frac{n}{2}\right\rfloor}(1-xq^i)}
  =\frac{n!}{2^{\left\lfloor\frac{n}{2}\right\rfloor}}\sum\limits_{\{\lambda \in \overline{{\cal P}}: \lambda_1\leqslant \left\lfloor\frac{n}{2}\right\rfloor\}}
  q^{|\lambda|}x^{\ell(\lambda)},
$$
and
$$
 \frac{ \sum \limits_{\sigma \in S_n} q^{\emaj(\sigma)} x^{\edes(\sigma)}}{\prod\limits_{i=1}^{\left\lfloor\frac{n-1}{2}\right\rfloor}(1-xq^i)}
  =\frac{n!}{2^{\left\lfloor\frac{n-1}{2}\right\rfloor}}\sum\limits_{\{\lambda \in \overline{{\cal P}}: \lambda_1\leqslant \left\lfloor\frac{n-1}{2}\right\rfloor\}}
  q^{|\lambda|}x^{\ell(\lambda)}.
$$
\end{cor}
\begin{proof}
  By Theorem \ref{odd eulerian} we have that $$ \frac{ \sum \limits_{\sigma \in S_n} q^{\omaj(\sigma)} x^{\odes(\sigma)}}{\prod\limits_{i=1}^{\left\lfloor\frac{n}{2}\right\rfloor}(1-xq^i)}
  =\frac{n!}{2^{\left\lfloor\frac{n}{2}\right\rfloor}}\frac{\prod\limits_{i=1}^{\left\lfloor\frac{n}{2}\right\rfloor}(1+xq^i)}
  {\prod\limits_{i=1}^{\left\lfloor\frac{n}{2}\right\rfloor}(1-xq^i)}$$ and the result follows immediately. The proof of the second equation is identical.
\end{proof}

Note that Corollary \ref{corollario over} can also be stated in terms of super-Schur functions.
Given a partition $\lambda$ and variables $x_1,...,x_m,y_1,...,y_n$ we denote by $s_{\lambda}(x_1,...,x_m/y_1,...,y_n)$ the super-Schur function (also known as hook Schur function, see \cite{BereleRegev})
associated to $\lambda$ (we refer the reader to \cite{Sta-Liesuper} for the definition and further information about super-Schur functions).
Then we have, by \cite[Equation (6)]{Brenti-superSchur},
$$
 \frac{ \sum \limits_{\sigma \in S_n} q^{\omaj(\sigma)} x^{\odes(\sigma)}}{\prod\limits_{i=1}^{\left\lfloor\frac{n}{2}\right\rfloor}(1-xq^i)}
  =\frac{n!}{2^{\left\lfloor\frac{n}{2}\right\rfloor}}\sum\limits_{k \geqslant 0}s_{(k)}(q,...,q^{\left\lfloor\frac{n}{2}\right\rfloor}/q,...,q^{\left\lfloor\frac{n}{2}\right\rfloor})x^k
$$
and
$$
 \frac{ \sum \limits_{\sigma \in S_n} q^{\emaj(\sigma)} x^{\edes(\sigma)}}{\prod\limits_{i=1}^{\left\lfloor\frac{n-1}{2}\right\rfloor}(1-xq^i)}
  =\frac{n!}{2^{\left\lfloor\frac{n-1}{2}\right\rfloor}}\sum\limits_{k \geqslant 0}s_{(k)}(q,...,q^{\left\lfloor\frac{n-1}{2}\right\rfloor}/q,...,q^{\left\lfloor\frac{n-1}{2}\right\rfloor})x^k.
$$

A second corollary of Theorem \ref{odd eulerian} is the odd-even analogue of the Gessel-Simion Theorem.
Recall our definition of $p_n$ from \S 2.
\begin{cor}
Let $n\in \mathbb{P}$. Then
$$
  \sum \limits_{\sigma \in S_n} (-1)^{\ell(\sigma)} q^{\omaj(\sigma)}
  =\left\lfloor\frac{n}{2}\right\rfloor!
  \prod\limits_{i=1}^{\left\lfloor\frac{n}{2}\right\rfloor} (1-q^{i}),
$$
and
$$
\sum \limits_{\sigma \in S_n} (-1)^{\ell(\sigma)} q^{\emaj(\sigma)} =
p_{n+1}\left\lfloor\frac{n}{2}\right\rfloor ! \prod\limits_{i=1}^{\left\lfloor\frac{n-1}{2}\right\rfloor} (1-q^{i}).
$$
\end{cor}

A further corollary is the following.
\begin{cor}
Let $n\in \mathbb{P}$. Then
$\sum \limits_{\sigma \in S_n} q^{\omaj(\sigma)}$ and
$\sum \limits_{\sigma \in S_n} q^{\emaj(\sigma)}$ are symmetric unimodal polynomials.
\end{cor}
\begin{proof}
It follows immediately from Theorem \ref{odd eulerian} that
$
  \sum \limits_{\sigma \in S_n} q^{\omaj(\sigma)}
  =\frac{n !}{2^{\lfloor\frac{n}{2}\rfloor}}
  \prod_{i=1}^{\left\lfloor\frac{n}{2}\right\rfloor} (1+q^{i}),
$
and that
$
\sum \limits_{\sigma \in S_n} q^{\emaj(\sigma)} =
\frac{n !}{2^{\lfloor\frac{n-1}{2}\rfloor}}
\prod_{i=1}^{\left\lfloor\frac{n-1}{2}\right\rfloor} (1+q^{i}).
$
But it is well known (see, e.g., \cite{Sta-Uni}) that the polynomial $\prod\limits_{i=1}^{k} (1+q^{i})$
is unimodal for all $k \geq 1$.
\end{proof}

Another unimodality result that arises from Corollary \ref{corollario over} is the following.

\begin{pro} \label{unimod over}
  Let $n,m \geqslant 1$. Then the polynomial $\sum\limits_{\{\lambda \in \overline{{\cal P}}: \lambda_1\leqslant n, \ell(\lambda)=m\}}
  q^{|\lambda|}$  is symmetric and unimodal with center of symmetry at $\frac{m(n+1)}{2}$.
\end{pro}
\begin{proof}
	Let $P_{n,m}(q):=\sum\limits_{\{\lambda \in \overline{{\cal P}}: \lambda_1\leqslant n, \ell(\lambda)=m\}}
  q^{|\lambda|}$.   It is easy to see that $\deg(P_{n,m})=nm$ and that $q^{(n+1)m}P_{n,m}(q^{-1})=P_{n,m}(q)$. Moreover, by Corollary \ref{corollario over} and Theorem \ref{odd eulerian}

  \begin{eqnarray*}
    \sum\limits_{m \geqslant 0} P_{n,m}(q)x^m &=& \frac{\prod\limits_{i=1}^{n}(1+xq^i)}
  {\prod\limits_{i=1}^{n}(1-xq^i)} \\
     &=& \left(\sum \limits_{j=0}^n e_j(q,q^2,...,q^n) x^j\right) \left(\sum \limits_{r\geqslant 0} h_r(q,q^2,...,q^n) x^r \right) \\
     &=& \sum \limits_{m\geqslant 0} \left(\sum \limits_{i=0}^m  e_i(q,q^2,...,q^n)h_{m-i}(q,q^2,...,q^n) \right) x^m \\
     &=& \sum \limits_{m\geqslant 0} \left(\sum \limits_{i=0}^m  e_i(1,q,...,q^{n-1})h_{m-i}(1,q,...,q^{n-1}) \right)q^m x^m,
  \end{eqnarray*}  where $e_i$ are the elementary symmetric functions and $h_i$ are the complete symmetric functions (see, e.g., \cite[Chapter 1]{Macdonald}).
  It is well known that $e_i(1,q,...,q^{n-1})=q^{\binom{i}{2}}\binom{n}{i}_q$, while $h_{m-i}(1,q,...,q^{n-1})=\binom{n+m-i-1}{m-i}_q$ (see, e.g., \cite[Example 1.3]{Macdonald}) where
  $\binom{a}{b}_q:= [a]_q!/([b]_q![a-b]_q)!$ is the $q$-binomial coefficient.
  Since $\binom{a}{b}_q$ is a symmetric unimodal polynomial of degree $b(a-b)$ (see, e.g., \cite[Theorem 11]{Sta-Uni}) and the product of two symmetric unimodal polynomials with nonnegative coefficients
  is again symmetric and unimodal (see, e.g. \cite[Proposition 1]{Sta-Uni}), we have that the product $e_i(1,q,...,q^{n-1})h_{m-i}(1,q,...,q^{n-1})$ is symmetric and unimodal with center of symmetry $\frac{m(n-1)}{2}$. Therefore $P_{n,m}(q)$ is symmetric and unimodal with center of symmetry $\frac{m(n+1)}{2}$.
 \end{proof}
 Note that Proposition \ref{unimod over} is related to, but different from, \cite[Conjecture 7.2]{Dousse Byungchan} namely that the polynomial $\sum\limits_{\{\lambda \in \overline{{\cal P}}: \lambda_1\leqslant n, \ell(\lambda)\leqslant m\}}q^{|\lambda|}$ is unimodal.

Note that the parabolic analogues of the odd and even $q$-Eulerian polynomials  don't factor nicely,
in general. For example, one can check that
$\sum_{\sigma \in S^{\{2\}}_4} q^{\omaj(\sigma)}=5q^3 +3q^2 + 3q + 1)$,
and that
$\sum_{\sigma \in S^{\{1,3\}}_5} q^{\emaj(\sigma)}=16q^3 + 4q^2 + 9q + 1$.
Also, the bivariate generating function of
$\omaj$ and $\emaj$ does not seem to factor. For example,
$\sum_{\sigma \in S_{3}} q_1^{\omaj(\sigma)} q_2^{\emaj(\sigma)}=q_1 q_2 + 2q_1 + 2q_2 + 1$.
The ``signed'' generating functions,
however, can always be reduced to that over a certain subset, which, in turn, can
often, though not always, be computed combinatorially.

For $n \in \mathbb{P}$ we define:
\[
D(S_n) := \{ \sigma \in S_n : | \sm(i)-\sm(i^{\ast})| \leq 1 \; \mbox{if} \;i \in [n-1]  \}.
\]
So, for example, $21534 \in D(S_5)$ while $23541 \notin D(S_5)$. We call the elements of $D(S_n)$
\textit{domino permutations}. It is not hard to see that
$|D(S_n)| = 2^{\left\lfloor \frac{n}{2} \right\rfloor} \left\lceil \frac{n}{2} \right\rceil !$ and that, if $n$ is even,
$D(S_n) = \{ \sigma \in S_n : | \sigma(i)-\sigma(i^{\ast})| \leq 1 \; \mbox{if} \;i \in [n-1]  \}$.

The proposition below gives a sign-reversing involution that reduces the computation of the signed
generating function over any quotient to the corresponding set of domino permutations.

\begin{pro} \label{domino-odd}
Let $n \in \mathbb{P}$, and $J \subseteq [n-1]$. Then
\[
\sum_{\sigma \in S_{n}^{J}} (-1)^{\ell(\sigma)}q_{1}^{\omaj(\sigma)} q_{2}^{\emaj(\sigma)} x_{1}^{\odes(\sigma)} x_{2}^{\edes(\sigma)}=
\sum_{\sigma \in D(S_n)^J} (-1)^{\ell(\sigma)} q_{1}^{\omaj(\sigma)} q_{2}^{\emaj(\sigma)} x_{1}^{\odes(\sigma)} x_{2}^{\edes(\sigma)}.
\]
\end{pro}
\begin{proof}
Let $\sigma \in S_{n}^{J} \setminus D(S_n)$ and
$r := \min \{ i \in [n-1] : | \sm(i)-\sm(i^{\ast})| \geq 2 \}$.
Define the map $\iota: S_{n}^{J} \setminus D(S_n) \rightarrow S_{n}^{J} \setminus D(S_n)$ by $\iota (\sigma):= (i, i^{\ast}) \, \sigma$, for all $\sigma \in S_{n}^{J} \setminus D(S_n)$. Then
$ \ell (\iota (\sigma)) \equiv \ell(\sigma)+1 \pmod{2}$, $D(\iota(\sigma)) =D(\sigma)$,
and $\iota (\iota (\sigma)) = \sigma$ for all $\sigma \in S_n^J \setminus D(S_n)$,
so the result follows.
\end{proof}


Note that there is a bijection between $D(S_{2m})$ and $S_m \times {\cal P}([m])$ obtained by associating to each $\sigma \in S_m$ and $S \subseteq [m]$ the permutation $u \in D(S_{2m})$ defined by
$$
u(2j-1):=
  \left\{
     \begin{array}{ll}
       2 \sigma(j)-1, & \hbox{if $j \notin S$,} \\
       2 \sigma(j), & \hbox{otherwise,}
     \end{array}
   \right.
 $$
and
$$
u(2j):=
  \left\{
     \begin{array}{ll}
       2 \sigma(j), & \hbox{if $j \notin S$,} \\
       2 \sigma(j)-1, & \hbox{otherwise,}
     \end{array}
   \right.
 $$
for $j \in [m]$.
So, for example, if $\sigma=4213$ and $S=\{ 2,3 \} $ then $u=78432156 $.
For this reason we identify these two sets and write $u=(\sigma,S)$
to mean that $u$ and $(\sigma,S)$ correspond under this bijection.
Also, note that if $J \subseteq [2m-1]$ then the bijection just
described restricts to a bijection between
$D(S_{2m}) \cap S_{2m}^J$ and
$ S_m^{J_e/2} \times {\cal P}(\{ i \in [m] : 2i-1 \notin J \}) $.

The proof of the following result is a routine check, and is therefore omitted.
\begin{lem}
\label{Atranslate}
Let $m \in \PP$, and $u \in D_{2m}$, $u=(\sigma, S)$.
Then
\begin{enumerate}
  \item $\ell(u)=4\ell(\sigma) + |S|$,
  \item $\odes(u)=|S|$, $\edes(u)=\des(\sigma)$,
  \item $\omaj(u)=\sum\limits_{t \in S}t$, $\emaj(u)=\maj(\sigma)$.
\end{enumerate}
\end{lem}

The following result is a refinement, and extension, of \cite[Theorem 1]{Wachs}
(which is the case $J= \emptyset$, $q_1=q_2=q^2$, $x_1=x/q$ and $x_2=x$).

\begin{thm}
\label{teorema-odd}
Let $m \in \mathbb{P}$, and $J \subseteq [2m-1]$.
Then
\[
\sum_{\sigma \in S_{2m}^{J}} (-1)^{\ell(\sigma)}q_{1}^{\omaj(\sigma)} q_{2}^{\emaj(\sigma)} x_{1}^{\odes(\sigma)} x_{2}^{\edes(\sigma)}=
\prod_{\{ i \in [m] : 2i-1 \notin J \}} (1-x_1 q_1^{i}) \;
\sum_{ \tau \in S_{m}^{ J_e/2 }} q_{2}^{\maj(\tau)} x_{2}^{\des(\tau)}.
\]

\end{thm}
\begin{proof}

Let $A:=\{ i \in [m] \, : \, 2i-1 \notin J \}$. Then by Lemma \ref{Atranslate} and the considerations
preceding it we have that
\begin{eqnarray*}
  \sum_{\sigma \in D^J_{2m}} (-1)^{\ell(\sigma)} q_{1}^{\omaj(\sigma)} q_{2}^{\emaj(\sigma)}x_{1}^{\odes(\sigma)} x_{2}^{\edes(\sigma)} &=&\sum_{\tau \in S_{m}^{J_e/2}}
\sum_{T \subseteq A} (-1)^{4 \ell(\tau) + |T|} q_{1}^{\sum_{t \in T} t} q_{2}^{\maj(\tau)} x_{1}^{|T|} x_{2}^{\des(\tau)}\\
&=&\sum_{ \tau \in S_{m}^{ J_e/2 }} q_{2}^{\maj(\tau)} x_{2}^{\des(\tau)} \prod_{a\in A} (1- x_1 q_1^{a}).
\end{eqnarray*}

The result follows from Proposition \ref{domino-odd}.
\end{proof}

We note the following consequences of Theorem \ref{teorema-odd}.

\begin{cor}
Let $m\in \mathbb{P}$. Then

  \[
\sum_{\sigma \in S_{2m}} (-1)^{\ell(\sigma)}q_{1}^{\omaj(\sigma)} q_{2}^{\emaj(\sigma)}=
[m]_{q_2}!\prod_{ i =1}^m (1- q_1^{i}).
\]
\end{cor}
Note that for symmetric groups of odd rank the bivariate signed generating function of $\omaj$ and $\emaj$ does not factorize nicely. For example,  $\sum_{\sigma \in S_5} (-1)^{\ell(\sigma)}q_{1}^{\omaj(\sigma)} q_{2}^{\emaj(\sigma)}=(1+y^2)(1+x^3+xy^4+x^4y^4-2x^3y^2-2xy^2)$.

\begin{cor}
\label{eulerian}
Let $m\in \mathbb{P}$ and $J\subseteq [2m-1]$. Then
\[
\sum_{\sigma \in S_{2m}^{J}} (-1)^{\ell(\sigma)}q^{\maj(\sigma)} x^{\des(\sigma)} =
\prod_{\{ i \in [m] : 2i-1 \notin J \}} (1-x q^{2i-1}) \;
\sum_{ \tau \in S_{m}^{ J_e/2 }}
q^{2 \maj(\tau)} x^{\des(\tau)}.
\]
\end{cor}
\begin{proof}
This follows immediately by taking $q_1=q_2=q^2$, $x_1 = \frac{x}{q}$, and $x_2=x$ in
Theorem \ref{teorema-odd}.
\end{proof}

\begin{cor}
Let $m\in \mathbb{P}$. Then
$$
\sum_{\sigma \in S_{2m}^{[2m]_o}} (-1)^{\ell(\sigma)}
q_{1}^{\omaj(\sigma)} q_{2}^{\emaj(\sigma)} x_{1}^{\odes(\sigma)} x_{2}^{\edes(\sigma)}
=\sum_{ \tau \in S_{m}} q_{2}^{\maj(\tau)} x_{2}^{\des(\tau)}. \Box
$$
\end{cor}

We note that the previous ``sign-balance'' identities are examples
of the phenomenon described in \cite{EFPT}, namely that the signed enumeration on
$2m$ objects by certain statistics is essentially equivalent to the ordinary enumeration on $m$ objects
by the same statistics.

\section{Types $B$ and $D$}
\label{B}

In this section we define and study odd and even analogues of the descent and flag-major statistics
on the classical Weyl groups of types $B$ and $D$, and compute the generating functions of these
statistics twisted by the one-dimensional characters of these groups.

 For $i\in [n-1]$ we write, for brevity, ``$s_i$'' rather than ``$s^B_i$''.
We also define
$$
s_i^*:=(i, i^*)(-i,-i^*)
$$
for all $i\in [n-1]$. Note that the involution of $B_n$ defined by $\sigma \mapsto s_i^*\sigma$,
restricts to $D_n$.

We define six statistics on the hyperoctahedral group $B_n$  by letting
$$
  \omaj(\sigma) := \sum_{i \in D(\sigma)_o} \frac{i+1}{2},
  ~ \emaj(\sigma):=\sum_{i \in D(\sigma)_e} \frac{i}{2},$$
 $$ \odes(\sigma) := |D(\sigma)_o|, ~ \edes(\sigma) := |D(\sigma)_e|,$$
 $$ \oneg(\sigma) := |\Neg(\sigma)_o|,~ \eneg(\sigma) := |\Neg(\sigma)_e|,
$$
 for all $\sigma \in B_n$.
So, for example, if $\sigma=[-2,5,3,1,-4]$
then $D(\sigma)=\{ 0,2,3,4 \}$, $\Neg(\sigma)=\{ 1,5 \}$,
$\omaj(\sigma)=2$, $\odes(\sigma)=1$, $\oneg(\sigma)=2$, $\emaj(\sigma)=3$, $\edes(\sigma)=3$, and $\eneg(\sigma)=0$.
Note that, if $\sigma \in S_n$, then the first four of these statistics coincide with those already defined in the previous section by the same name.

We then define the {\em odd flag-major index} and the {\em even flag-major index} of $\sigma \in B_n$
by letting
$$
\ofmaj(\sigma) := 2 \omaj(\sigma) + \oneg(\sigma), ~~\efmaj(\sigma) := 2 \emaj(\sigma) + \eneg(\sigma).
$$
So, if $\sigma$ is as above then $\ofmaj(\sigma)=\efmaj(\sigma)=6$.

For $\sigma \in B_n$ we let $|\sigma|_n := [\sigma(1), \ldots , \sigma(n-1), |\sigma(n)|]$.
We then define six more statistics on $B_n$ by letting

$$ \oDmaj(\sigma) :=  \ofmaj(|\sigma|_n), ~\eDmaj(\sigma) :=  \efmaj(|\sigma|_n),$$
$$\oneg_D(\sigma) := \oneg(|\sigma|_n),  ~ \eneg_D(\sigma) := \eneg(|\sigma|_n),$$
$$ \odes_D(\sigma) := \odes(|\sigma|_n), ~ \edes_D(\sigma) := \edes(|\sigma|_n),
$$
for all $\sigma \in B_n$.
So, for example, if $\sigma$ is as above
then $\odes_D(\sigma)=1$, $\oDmaj(\sigma)=5$, $\edes_D(\sigma)=2$ and $\eDmaj(\sigma)=2$. We call $\oDmaj$ and $\eDmaj$ the
{\em odd D-major index} and the {\em even D-major index} of $\sigma \in B_n$, respectively.

Our aim is to compute the generating functions
$ \sum\limits_{\sigma \in B_n} \chi(\sigma) x^{\ofmaj(\sigma)} y^{\odes(\sigma)} z^{\oneg(\sigma)}$ and
$ \sum\limits_{\sigma \in D_n} \chi(\sigma) x^{\oDmaj(\sigma)} y^{\odes_D(\sigma)} z^{\oneg_D(\sigma)}$
where $\chi$ is any one-dimensional character of the corresponding group,
and the analogue even ones.

\subsection{The trivial character}

We start with the trivial characters of $B_n$ and $D_n$.
Recall that we let $ p_n := (1+(-1)^n)/2$.
\begin{thm}
\label{trivialB}
Let $n \geq 2$. Then
$$
\sum_{\sigma \in B_n}x^{\ofmaj(\sigma)} y^{\odes(\sigma)} z^{\oneg(\sigma)} =
                   \displaystyle
                   \frac{n!}{2^{\lfloor \frac{n}{2}  \rfloor}} (1+xz)^{p_{n+1}}
                   \prod_{j=1}^{\lfloor \frac{n}{2}  \rfloor}
                   (1+3xz+3yx^{2j}+yzx^{2j+1}),
$$
and
$$
\sum_{\sigma \in B_n}x^{\efmaj(\sigma)} y^{\edes(\sigma)} z^{\eneg(\sigma)} =
                   \displaystyle
                   \frac{n!}{2^{\lfloor \frac{n-1}{2}  \rfloor}} (1+y) (1+xz)^{p_n}
                   \prod_{j=1}^{\lfloor \frac{n-1}{2}  \rfloor}
                   (1+3xz+3yx^{2j}+yzx^{2j+1}).
$$
\end{thm}
\begin{proof}
We only prove the formula for the odd statistics, the proof for the even ones being
analogous. We proceed by induction on $n \geq 2$, the result being easy to check for $n=2$.
Suppose first that $n$ is odd. Then we have that
\begin{eqnarray*}
\sum_{\sigma \in B_n}x^{\ofmaj(\sigma)} y^{\odes(\sigma)} z^{\oneg(\sigma)}
& = & \sum_{i \in [\pm n]} \sum_{\{ \sigma \in B_{n} : \sigma(n)=i \}}
  x^{\ofmaj(\sigma)} y^{\odes(\sigma)} z^{\oneg(\sigma)} \\
& = & \sum_{\tau \in B_{n-1}} \sum_{i \in [n]} (1+zx) x^{\ofmaj(\tau)} y^{\odes(\tau)} z^{\oneg(\tau)}\\
& = & n(1+zx) \sum_{\tau \in B_{n-1}} x^{\ofmaj(\tau)} y^{\odes(\tau)} z^{\oneg(\tau)},
\end{eqnarray*} as desired.
Suppose now that $n$ is even.
Let $\sigma \in B_n$ and $[\tau,i,j]$ be its window notation ($i,j \in [\pm n]$).
Then one can check that
\[
\ofmaj(\sigma)= \left\{
                 \begin{array}{ll}
                   \ofmaj(\tau), & \hbox{if $0< i< j$,} \\
                   \ofmaj(\tau)+1, & \hbox{if $i<0$ and $i<j$,} \\
                   \ofmaj(\tau)+n+1, & \hbox{if $0 > i>j$,} \\
                   \ofmaj(\tau)+n, & \hbox{if $i>0$ and $i>j$.} \\
                 \end{array}
               \right.
\]
Therefore we conclude that
\begin{eqnarray*}
 \sum \limits_{\sigma \in B_{n}} x^{\ofmaj(\sigma)} y^{\odes(\sigma)} z^{\oneg(\sigma)}
 &=& \sum_{\substack{i,j \in [\pm n]\\ i \neq \pm j}}
 \sum \limits_{\substack{\sigma \in B_n \\ \sigma(n-1)=i,\, \sigma(n)=j}}
 x^{\ofmaj(\sigma)} y^{\odes(\sigma)} z^{\oneg(\sigma)} \\
 &=& \sum_{0< i<j} \sum \limits_{\tau \in B_{n-2}} x^{\ofmaj(\tau)} y^{\odes(\sigma)} z^{\oneg(\tau)} \\
 &+& \sum_{\substack{i<j \\ i<0 \\ i \neq -j}} \sum \limits_{\tau \in B_{n-2}} x^{\ofmaj(\tau)+1} y^{\odes(\tau)} z^{\oneg(\tau)+1}\\
 &+& \sum_{0> i>j} \sum \limits_{\tau \in B_{n-2}} x^{\ofmaj(\tau)+n+1} y^{\odes(\tau)+1} z^{\oneg(\tau)+1} \\
 &+& \sum_{\substack{i>j \\ i>0 \\ i \neq -j}} \sum \limits_{\tau \in B_{n-2}} x^{\ofmaj(\tau)+n} y^{\odes(\tau)+1} z^{\oneg(\tau)}\\
\end{eqnarray*}
$$
 =\binom{n}{2} (1 +3z x + z y x^{n+1} + 3 y x^{n})
   \sum \limits_{\tau \in B_{n-2}} x^{\ofmaj(\tau)} y^{\odes(\tau)} z^{\oneg(\tau)},
$$
and the result again follows.
\end{proof}

The corresponding result for $D_n$ is a consequence of the one in type $B$.
Note that if $f_1,...,f_k : B_n \rightarrow \mathbb{N}$ then
\begin{equation}
\label{relateBD}
 \sum_{\sigma \in B_n}x_1^{f_1(|\sigma|_n)}\cdots x_k^{f_k(|\sigma|_n)} =
 2\sum_{\sigma \in D_n}x_1^{f_1(|\sigma|_n)}\cdots x_k^{f_k(|\sigma|_n)}.
\end{equation}

\begin{thm}
\label{trivialD}
Let $n \geqslant 3$. Then
$$
\sum_{\sigma \in D_n}x^{\oDmaj(\sigma)}y^{\odes_D(\sigma)}z^{\oneg_D(\sigma)} =
\displaystyle \frac{n!}{2^{\lfloor \frac{n}{2}  \rfloor}} (1+2zx+yx^{n})^{p_n}
\prod_{j=1}^{\lfloor \frac{n-1}{2}  \rfloor }(1+3xz+3yx^{2j}+yzx^{2j+1})
$$
and
$$
\sum_{\sigma \in D_n}x^{\eDmaj(\sigma)}y^{\edes_D(\sigma)}z^{\eneg_D(\sigma)}=
\displaystyle \frac{n!}{2^{\lfloor \frac{n-1}{2}  \rfloor}} (1+y)
 \left(1+2zx+yx^{n}\right)^{p_{n+1}} \prod_{j=1}^{\lfloor \frac{n-2}{2}  \rfloor}
 (1+3zx+3yx^{2j}+yzx^{2j+1}).
$$
\end{thm}
\begin{proof}
Let $n \geq 3$ be odd.
Then
\begin{eqnarray*}
\sum_{ \sigma \in B_n}x^{\oDmaj(\sigma)}y^{\odes_D(\sigma)}z^{\oneg_D(\sigma)}
 &=& \sum \limits_{i\in [\pm n]}\sum_{\{ \sigma \in B_{n} : \, \sigma(n)=i \} }x^{\ofmaj(|\sigma|_n)}y^{\odes(|\sigma|_n)}z^{\oneg(|\sigma|_n)} \\
 &=& 2n \sum_{\tau \in B_{n-1}}x^{\ofmaj(\tau)}y^{\odes(\tau)}z^{\oneg(\tau)},
\end{eqnarray*}
and the result follows from Theorem \ref{trivialB} and (\ref{relateBD}).
Let now $n$ be even. Let $\sigma \in B_n$ and $[\tau,i,j]$ be its window notation ($i,j \in [\pm n]$).
Then we have that
\[
\oDmaj(\sigma)= \left\{
                 \begin{array}{ll}
                   \ofmaj(\tau), & \hbox{if $0< i< |j|$,} \\
                   \ofmaj(\tau)+1, & \hbox{if $i<0$,} \\
                   \ofmaj(\tau)+n, & \hbox{if $i>|j|$.} \\
                 \end{array}
               \right.
\]
Therefore
\begin{eqnarray*}
  \sum_{\sigma \in B_n}x^{\oDmaj(\sigma)}y^{\odes_D(\sigma)}z^{\oneg_D(\sigma)}
&=& \sum \limits_{0<i<|j|}\sum_{\tau \in B_{n-2}}x^{\ofmaj(\tau)}y^{\odes(\tau)}z^{\oneg(\tau)} \\
&&+\sum \limits_{\substack{i<0 \\ -i \neq |j|}}\sum_{\tau \in B_{n-2}}x^{\ofmaj(\tau)+1}y^{\odes(\tau)}z^{\oneg(\tau)+1}\\
&&+\sum \limits_{i>|j|}\sum_{\tau \in B_{n-2}}x^{\ofmaj(\tau)+n}y^{\odes(\tau)+1}z^{\oneg(\tau)}
\end{eqnarray*} $$ =n(n-1)\left(1+2zx+yx^{n}\right)\sum_{\tau \in B_{n-2}}x^{\ofmaj(\tau)}y^{\odes(\tau)}z^{\oneg(\tau)},
$$
and the result again follows from Theorem \ref{trivialB} and (\ref{relateBD}).

The proof for the even statistics is analogous and is therefore omitted.
\end{proof}

\subsection{The alternating character}

The computation for the character $\sigma \mapsto (-1)^{\ell_B(\sigma)}$
 of $S^B_n$ is considerably more involved.
We begin with the following reduction result.
For $n \in \PP$ we let
$$
D(B_n):= \{ \sigma \in B_n : |\sigma^{-1}(i) - \sigma^{-1}(i^{\ast}) | \leq 1  \, for \;  all \; i \in [n-1] \}.
$$
We call the elements of $D(B_n)$ {\em signed domino permutations}. So, for example, $[-3,-4,5,2,1] \in D(B_5)$ while
$ [3,-4,1,2] \notin D(B_4)$.

\begin{pro}
\label{Bdominored}
Let $n \in \PP$, and $S \subseteq [n]$. Then
$$
\sum_{\{ \sigma \in B_n : \Neg(\sigma)=S \}} (-1)^{\ell_B (\sigma)}
x_1^{\omaj(\sigma)} x_2^{\emaj(\sigma)} y_1^{\odes(\sigma)} y_2^{\edes(\sigma)}
$$
$$ =
\sum_{\{ \sigma \in D(B_n)  : \Neg(\sigma)=S \}} (-1)^{\ell_B (\sigma)}
x_1^{\omaj(\sigma)} x_2^{\emaj(\sigma)} y_1^{\odes(\sigma)} y_2^{\edes(\sigma)}.
$$
\end{pro}
\begin{proof}
Let $\widehat{B_n} \eqdef \{ \sigma \in B_n : \Neg(\sigma)=S \}$, and $\varphi : \widehat{B_n} \setminus D(B_n) \rightarrow
\widehat{B_n} \setminus D(B_n)$ be defined by $\varphi(\sigma):=  s_r^*\sigma$ where $r:= \min\{ i \in [n-1] : |\sigma^{-1}(i) - \sigma^{-1}(i^{\ast}) | \geq 2 \}$ if $\sigma \in \widehat{B_n} \setminus D(B_n)$.
Then $\varphi : \widehat{B_n} \setminus D(B_n) \rightarrow
\widehat{B_n} \setminus D(B_n)$ is an involution, $\ell_B( \varphi(\sigma)) \equiv \ell_B(\sigma) +1
\pmod{2}$, and $D(\varphi(\sigma))=D(\sigma)$, so the result follows.
\end{proof}

Note that there is a bijection between $D(B_{2m})$ and $B_m \times {\cal P}([m])$ obtained by associating to each $\sigma \in B_m$ and $S \subseteq [m]$ the signed permutation $u \in B_{2m}$ defined by
\[
(u(2j-1),u(2j)):=
\left\{
\begin{array}{ll}
       (2 \sigma(j), 2 \sigma(j)-1), & \hbox{if $j \in S, \sigma(j)>0$,} \\
       (2 \sigma(j)-1, 2 \sigma(j)), & \hbox{if $j \notin S, \sigma(j)>0$,} \\
       (2 \sigma(j)+1, 2 \sigma(j)), & \hbox{if $j \in S, \sigma(j)<0$,} \\
       (2 \sigma(j), 2 \sigma(j)+1), & \hbox{if $j \notin S, \sigma(j)<0$,}
\end{array}
\right.
\]
for $j \in [m]$.
So, for example, if $\sigma=[3,-1,-2,5,-4]$ and $S = \{ 1,4,5 \}$
then $u=[6,5,-2,-1,-4,-3,10,9,-7,-8]$.
Because of this bijection we will often identify $D(B_{2m})$ and $B_m \times {\cal P}([m])$ and write simply $u=(\sigma, S)$ to mean that $u$ and $(\sigma, S)$ correspond under this bijection.
The proof of the next result is a routine check using our definitions,
and is therefore omitted.

\begin{lem}
\label{Btranslate}
Let $m \in \PP$, and $u \in D(B_{2m})$, $u=(\sigma, S)$.
Then
\begin{enumerate}
  \item $\negg(u)=2 \negg(\sigma)$, $\oneg(u)=\eneg(u)=\negg(\sigma)$;
  \item $\odes(u)=|S|$,  $\edes(u)=\des(\sigma)$;
  \item $\omaj(u)=\sum_{t \in S} t$, $\emaj(u)=\maj(\sigma)$;
  \item $\ell_A(u)=4 \ell_A(\sigma) + |S|$ and $\ell_B(u)=4 \ell_B(\sigma) + |S|- \negg(\sigma)$.
\end{enumerate}

\end{lem}

For shortness, we define the following two monomials, for any $\sigma \in B_n$:
$ \sigma_o(x,y):=x^{\omaj(\sigma)} y^{\odes(\sigma)}$ and $\sigma_e(x,y):=x^{\emaj(\sigma)} y^{\edes(\sigma)}$.
We need to prove a further reduction result.

\begin{lem}
\label{Bmaxred}
Let $m \in \PP$, and $S \subseteq [2m+1]$. Then
$$
\sum_{\{ \sigma \in B_{2m+1} : \Neg(\sigma)=S \}} (-1)^{\ell_B(\sigma)}\sigma_o(x,y) =
\sum_{\{ \sigma \in B_{2m+1} : \Neg(\sigma)=S, |\sigma(2m+1)|=2m+1   \}}  (-1)^{\ell_B(\sigma)}\sigma_o(x,y),
$$ and
$$
\sum_{\{ \sigma \in B_{2m+1} : \Neg(\sigma)=S \}} (-1)^{\ell_B (\sigma)} \sigma_e(x,y) =
\sum_{\{ \sigma \in B_{2m+1} : \Neg(\sigma)=S, |\sigma(1)|=2m+1   \}} (-1)^{\ell_B (\sigma)}
\sigma_e(x,y).
$$
\end{lem}
\begin{proof}
Let $\psi : \{ \sigma \in B_{2m+1} : |\sigma(2m+1)| \neq 2m+1 \} \rightarrow \{ \sigma \in B_{2m+1} : |\sigma(2m+1)| \neq 2m+1 \}$ be defined by $\psi(\sigma):=  s_r^* \sigma$ where $r:= \sigma(2m+1)$.
Then $\psi$ is an involution, $\ell_B( \psi(\sigma)) \equiv \ell_B(\sigma) +1 \pmod{2}$, $\Neg(\psi(\sigma))=\Neg(\sigma)$ and $D(\sigma)_o=D(\psi(\sigma))_o$. This proves the
first equation. The proof for the even one is analogous is therefore omitted.

%
%
\end{proof}

For $S\subseteq [n]$ we define $S^*:=\{i^*:i \in S\}$.
Using the reductions proved so far we can now compute explicitly several ``building blocks''
of the generating functions that we are interested in.
\begin{lem}
\label{evenneg}
Let $m \in \PP$, and $S \subseteq [2m]$. Then
\[
\sum_{\{ \sigma \in B_{2m} : \Neg(\sigma)=S \}}
(-1)^{\ell_B(\sigma)}\sigma_o(x,y)=
\left\{
\begin{array}{ll}
       (-1)^{\frac{|S|}{2}} m! \displaystyle \prod_{i=1}^{m} (1-yx^i), & \hbox{if $S=S^*$,} \\
       0, & \hbox{otherwise,}
\end{array}
\right.
\]
and
$$
\sum_{\{ \sigma \in B_{2m} : \Neg(\sigma)=S \}}
(-1)^{\ell_B(\sigma)}\sigma_e(x,y)=0.
$$
\end{lem}
\begin{proof}
Note that, by Proposition \ref{Bdominored},
\[
\sum_{\{ \sigma \in B_{2m} : \Neg(\sigma)=S \}} (-1)^{\ell_B (\sigma)} \sigma_o(x,y) =
\sum_{\{ \sigma \in D(B_{2m})  : \Neg(\sigma)=S \}} (-1)^{\ell_B (\sigma)} \sigma_o(x,y),
\]
and the right hand side is zero whenever $S\neq S^*$. If $S=S^*$,
by Lemma \ref{Btranslate} we have that
\begin{eqnarray*}
\sum \limits_{\{ \sigma \in D(B_{2m}) : \Neg(\sigma)=S \}}
(-1)^{\ell_B(\sigma)}\sigma_o(x,y)
& = &
\sum_{\{ \sigma \in B_m : \Neg(\sigma)=S_e/2  \}}
\sum_{T \subseteq [m]}
(-1)^{|T|+\negg(\sigma)} x^{\sum\limits_{t \in T} t} y^{|T|} \\
& = &
\sum_{\{ \sigma \in B_m : \Neg(\sigma)=S_e/2  \}}
(-1)^{\negg(\sigma)}
\sum_{T \subseteq [m]}(-y)^{|T|}
 x^{\sum\limits_{t \in T} t}  \\
& = &
(-1)^{|S|/2} m! \prod_{i=1}^{m} (1-yx^i),
\end{eqnarray*} as claimed.

For the even statistics note that the assignment $\sigma \mapsto s_{|\sigma(1)|}^*\sigma$ defines an involution of $B_{2m}$ which preserves the functions $\emaj$
and $\edes$, so the result follows.
\end{proof}

\begin{lem}
\label{even-odd-NEG}
  Let $m\in \mathbb{P}$ and $S \subseteq [2m+1]$. Then
$$\sum_{\{ \sigma \in B_{2m+1} : \Neg(\sigma)=S \}} (-1)^{\ell_B (\sigma)} \sigma_e(x,y)=$$ $$=\left\{
                                                                                           \begin{array}{ll}
                                                                                             y(-1)^{\frac{|S|+1}{2}} m! \displaystyle \prod_{i=1}^{m} (1-yx^i), & \hbox{if $1\in S$ and $(S\setminus \{1\}-1)^*=S\setminus \{1\}-1$,} \\
                                                                                             (-1)^{\frac{|S|}{2}} m! \displaystyle \prod_{i=1}^{m} (1-yx^i), & \hbox{if $1\not\in S$ and $(S-1)^*=S-1$,} \\
                                                                                             0, & \hbox{otherwise.}
                                                                                           \end{array}
 \right.
$$
\end{lem}
\begin{proof}
  Let $1\in S$. Then $0 \in D(\sigma)$ for all $\sigma \in B_n$ such that $\Neg(\sigma)=S$. We have that, by Lemma \ref{Bmaxred},
\begin{eqnarray*}
  \sum_{\{ \sigma \in B_{2m+1} : \Neg(\sigma)=S \}} (-1)^{\ell_B (\sigma)} \sigma_e(x,y) &=&  
   \sum_{ \substack{\{\sigma \in B_{2m+1}: \\ \Neg(\sigma)=S,\, \sigma(1)=-2m-1 \}}} (-1)^{\ell_B (\sigma)} \sigma_e(x,y) \\
   &=&  -y\sum_{\substack{ \{\sigma \in B_{2m}: \\ \Neg(\sigma)=S\setminus\{1\}-1\}}} (-1)^{\ell_B (\sigma)} \sigma_o(x,y),
\end{eqnarray*} and the result follows by Lemma \ref{evenneg}.
If $1\not \in S$ the result follows by analogous computations.
\end{proof}

Note that if $|S|$ is even (resp., odd) the first (resp., second) case in the above lemma cannot occur.  We can now prove one of our main results. Recall that we have defined $p_n:=\frac{1+(-1)^n}{2}$.
\begin{thm}
\label{fourcorners}
Let $n \geqslant 2$. Then

\begin{equation} \label{new+}
\sum_{\substack{\{ \sigma \in B_{n}  :\, \sigma(n)>0 \\ \negg(\sigma) \equiv \epsilon \modue \}}}
(-1)^{\ell_B(\sigma)}x^{\omaj(\sigma)} y^{\odes(\sigma)}z_1^{\oneg(\sigma)}
z_2^{\eneg(\sigma)} = p_\epsilon \mathcal{O}_n(x,y,z_1,z_2),
\end{equation}

\begin{equation} \label{new-}
\sum_{\substack{\{ \sigma \in B_{n}  :\, \sigma(n)<0 \\ \negg(\sigma) \equiv \epsilon \modue \}}}
(-1)^{\ell_B(\sigma)}x^{\omaj(\sigma)} y^{\odes(\sigma)}z_1^{\oneg(\sigma)}
z_2^{\eneg(\sigma)} = -z_1z_2^{1-\epsilon}p_{n+\epsilon} \mathcal{O}_n(x,y,z_1,z_2),
\end{equation}

%
%
%
where $\epsilon \in \{0,1\}$ and $\mathcal{O}_n(x,y,z_1,z_2):=\left\lfloor \frac{n}{2} \right\rfloor !
    (1-z_1 z_2)^{\left\lfloor \frac{n-1}{2} \right\rfloor}
    \displaystyle \prod_{i=1}^{\left\lfloor \frac{n}{2} \right\rfloor} (1-yx^i)$.
\end{thm}

\begin{proof}
We write $\sigma_o(x,y,z_1,z_2):=x^{\omaj(\sigma)} y^{\odes(\sigma)}z_1^{\oneg(\sigma)} z_2^{\eneg(\sigma)}$ for shortness and
we begin by proving \eqref{new+} for $\epsilon=1$. Note that
\begin{eqnarray*}
\label{fc1}
\sum_{\{ \sigma \in B_{n} \setminus D_n : \sigma(n)>0 \}}
& (-1)^{\ell_B(\sigma)}& \sigma_o(x,y,z_1,z_2) = \\
& = &
\sum_{S \subseteq [n-1]} z_1^{|S_o|} z_2^{|S_e|}
\sum_{\{ \sigma \in B_{n} \setminus D_n : \Neg(\sigma)=S \}} (-1)^{\ell_B(\sigma)}\sigma_o(x,y) \\
& = &
\sum_{\{ S \subseteq [n-1] : |S| \equiv 1\modue \} }
z_1^{|S_o|} z_2^{|S_e|}
\sum_{\{ \sigma \in B_{n} : \Neg(\sigma)=S \}} (-1)^{\ell_B(\sigma)}\sigma_o(x,y).
\end{eqnarray*}
This already proves \eqref{new+} if $\epsilon=1$ and $n$ is even,
by Lemma \ref{evenneg}. Suppose now that $n$ is odd.
Then for all $S \subseteq [n-1]$ such that $|S| \equiv 1 \pmod{2}$
we have that, by Lemmas \ref{Bmaxred} and \ref{evenneg},
\begin{eqnarray*}
\sum_{\{ \sigma \in B_{n} : \Neg(\sigma)=S \}} (-1)^{\ell_B(\sigma)}\sigma_o(x,y)
& = &
\sum_{\{ \sigma \in B_{n} : \Neg(\sigma)=S, \sigma(n)=n \}} (-1)^{\ell_B(\sigma)}\sigma_o(x,y)
\\
& = &
\sum_{\{ \sigma \in B_{n-1} : \Neg(\sigma)=S \} } (-1)^{\ell_B(\sigma)}\sigma_o(x,y)=0,
\end{eqnarray*} and \eqref{new+} for $\epsilon=1$ again follows.

We now prove \eqref{new+} for $\epsilon=0$. Suppose first that $n$ is even.
Then proceeding as in the previous case and using Lemma \ref{evenneg}
we have that
\begin{eqnarray*}
\sum_{\{ \sigma \in D_n : \sigma(n)>0 \}}
&(-1)^{\ell_B(\sigma)}&\sigma_o(x,y,z_1,z_2)= \\
& = &
\sum_{\{ S \subseteq [n-1] : |S| \equiv 0 \modue \} }
z_1^{|S_o|} z_2^{|S_e|}
\sum_{\{ \sigma \in B_{n} : \Neg(\sigma)=S \}}
(-1)^{\ell_B(\sigma)}\sigma_o(x,y) \\
& = &
\sum_{\{ S \subseteq [n-1] : S=S^*\} }
z_1^{|S_o|} z_2^{|S_e|}
(-1)^{\frac{|S|}{2}} \left( \frac{n}{2} \right)! \displaystyle \prod_{i=1}^{n/2} (1-yx^i) \\
& = &
 \left( \frac{n}{2} \right)! \displaystyle \prod_{i=1}^{n/2}
(1-yx^i) \sum_{ \{S \subseteq [n-2] : S=S^* \} }
(-z_1 z_2)^{\frac{|S|}{2}},
\end{eqnarray*}
and \eqref{new+} follows in this case.
If $n$ is odd then by Lemma \ref{Bmaxred} we have similarly that
\begin{eqnarray*}
\sum_{\{ \sigma \in D_n : \sigma(n)>0 \}}
&(-1)^{\ell_B(\sigma)}&\sigma_o(x,y,z_1,z_2)= \\
& = &
\sum_{\{ S \subseteq [n-1] : |S| \equiv 0\modue \} }
z_1^{|S_o|} z_2^{|S_e|}
\sum_{\{ \sigma \in B_{n} : \Neg(\sigma)=S, \sigma(n)=n \}}
(-1)^{\ell_B(\sigma)}\sigma_o(x,y) \\
& = &
\sum_{\{ S \subseteq [n-1] : |S| \equiv 0\modue \} }
z_1^{|S_o|} z_2^{|S_e|}
\sum_{\{ \sigma \in B_{n-1} : \Neg(\sigma)=S \}}
(-1)^{\ell_B(\sigma)}\sigma_o(x,y) \\
& = &
\sum_{\{ S \subseteq [n-1] : S=S^* \} }
z_1^{|S_o|} z_2^{|S_e|}
(-1)^{\frac{|S|}{2}} \left( \frac{n-1}{2} \right)! \displaystyle \prod_{i=1}^{(n-1)/2} (1-yx^i) \\
& = &
\left\lfloor \frac{n}{2} \right\rfloor !
\displaystyle \prod_{i=1}^{\left\lfloor \frac{n}{2} \right\rfloor}
(1-yx^i) \sum_{T \subseteq [n-1]_o }
(-z_1 z_2)^{|T|},
\end{eqnarray*}
by Lemma \ref{evenneg} and \eqref{new+} again follows.

We now prove \eqref{new-} for $\epsilon=1$. We have that
\begin{eqnarray*}
\sum_{\{ \sigma \in B_{n} \setminus D_n : \sigma(n)<0 \}}
&(-1)^{\ell_B(\sigma)}&\sigma_o(x,y,z_1,z_2) \\
& =&
\sum_{\{ S \subseteq [n] : n \in S \} } z_1^{|S_o|} z_2^{|S_e|}
\sum_{\{ \sigma \in B_{n} \setminus D_n : \Neg(\sigma)=S \}} (-1)^{\ell_B(\sigma)}\sigma_o(x,y) \\
& =&
\sum_{\{ S \subseteq [n] : |S| \equiv 1\modue, n \in S \} }
z_1^{|S_o|} z_2^{|S_e|}
\sum_{\{ \sigma \in B_{n} : \Neg(\sigma)=S \}} (-1)^{\ell_B(\sigma)}\sigma_o(x,y),
\end{eqnarray*}
and this proves \eqref{new-} in this case by Lemma \ref{evenneg} if
$n$ is even. Suppose now that $n$ is odd. Then, from the previous equation
and Lemmas \ref{Bmaxred} and \ref{evenneg}, we have that

\begin{eqnarray*}
  \sum_{\{ \sigma \in B_{n} \setminus D_n : \sigma(n)<0 \}}&(-1)^{\ell_B(\sigma)}&\sigma_o(x,y,z_1,z_2) =  \\
  & = &
\sum_{\{ S \subseteq [n] : |S| \equiv 1\modue, n \in S \} }
z_1^{|S_o|} z_2^{|S_e|}
\sum_{\{ \sigma \in B_{n} : \Neg(\sigma)=S, \sigma(n)=-n \}}(-1)^{\ell_B(\sigma)}\sigma_o(x,y) \\
& = &
- \sum_{\{ T \subseteq [n-1] : |T| \equiv 0\modue \} }
z_1^{|T_o|+1} z_2^{|T_e|}
\sum_{\{ \sigma \in B_{n-1} : \Neg(\sigma)=T \}} (-1)^{\ell_B(\sigma)}\sigma_o(x,y) \\
& = &
- \sum_{\{ T \subseteq [n-1] : T=T^* \} }
z_1^{|T_o|+1} z_2^{|T_e|}
(-1)^{\frac{|T|}{2}} \left( \frac{n-1}{2} \right)!
\displaystyle \prod_{i=1}^{(n-1)/2} (1-yx^i) \\
& = &
-z_1 \sum_{\{ T \subseteq [n-1] : T=T^* \} }
(-z_1 z_2)^{\frac{|T|}{2}}
 \left( \frac{n-1}{2} \right)! \displaystyle \prod_{i=1}^{(n-1)/2} (1-yx^i) \\
& = &
-z_1 \left\lfloor \frac{n}{2} \right\rfloor!
\displaystyle \prod_{i=1}^{\left\lfloor \frac{n}{2} \right\rfloor}
(1-yx^i) \sum_{ S \subseteq [n-1]_o }
(-z_1 z_2)^{|S|}
\end{eqnarray*}
and \eqref{new-} follows if $\epsilon=1$.

Finally, if $n$ is even, then we obtain similarly that
\begin{eqnarray*}
\sum_{\{ \sigma \in D_n : \sigma(n)<0 \}}
&(-1)^{\ell_B(\sigma)}&\sigma_o(x,y,z_1,z_2)= \\
& = &
\sum_{\{ S \subseteq [n] : n \in S, |S| \equiv 0\modue \} }
z_1^{|S_o|} z_2^{|S_e|}
\sum_{\{ \sigma \in B_{n} : \Neg(\sigma)=S \}}
(-1)^{\ell_B(\sigma)}\sigma_o(x,y) \\
& = &
\sum_{\{ S \subseteq [n] : n \in S, S=S^* \} }
z_1^{|S_o|} z_2^{|S_e|}
(-1)^{\frac{|S|}{2}} \left( \frac{n}{2} \right)! \displaystyle \prod_{i=1}^{n/2} (1-yx^i) \\
& = &
 \left( \frac{n}{2} \right)! \displaystyle \prod_{i=1}^{n/2}
(1-yx^i) \sum_{T \subseteq [n-2]_o }
(-z_1 z_2)^{|T|+1}.
\end{eqnarray*}
and \eqref{new-}  follows for $\epsilon=0$.
If $n$ is odd then, by Lemmas \ref{Bmaxred} and \ref{evenneg}
\begin{eqnarray*}
\sum_{\{ \sigma \in D_n : \sigma(n)<0 \}}&
&(-1)^{\ell_B(\sigma)}\sigma_o(x,y,z_1,z_2)= \\
& = &
\sum_{\{ S \subseteq [n] : n \in S, |S| \equiv 0\modue \} }
z_1^{|S_o|} z_2^{|S_e|}
\sum_{\{ \sigma \in B_{n} : \Neg(\sigma)=S, \sigma(n)=-n \}}
(-1)^{\ell_B(\sigma)}\sigma_o(x,y)\\
& = &
- \sum_{\{ S \subseteq [n] : n \in S, |S| \equiv 0\modue \} }
z_1^{|S_o|} z_2^{|S_e|}
\sum_{\{ \tau \in B_{n-1} : \Neg(\tau)=S \setminus \{ n \} \}}
(-1)^{\ell_B(\tau)}\tau_o(x,y) \\
& = &
- \sum_{\{ T \subseteq [n-1] : |T| \equiv 1\modue \} }
z_1^{|T_o|+1} z_2^{|T_e|}
\sum_{\{ \tau \in B_{n-1} : \Neg(\tau)=T \}}
(-1)^{\ell_B(\tau)}\tau_o(x,y)= 0.
\end{eqnarray*}
This proves \eqref{new-} for $\epsilon=0$ and this concludes the proof.
\end{proof}

%
%

The following is the even analogue of the previous theorem. Its proof is similar and is therefore omitted.

\begin{thm}
\label{efourcorners}
Let $n\geqslant 2$. Then

\begin{equation}
\sum_{\substack{\{ \sigma \in B_{n}  :\, \sigma(n)>0 \\ \negg(\sigma) \equiv \epsilon \modue \}}}
(-1)^{\ell_B(\sigma)} x^{\emaj(\sigma)} y^{\edes(\sigma)}z_1^{\oneg(\sigma)}z_2^{\eneg(\sigma)} = \frac{(-yz_1)^{\epsilon}p_{n+1}}{1-z_1z_2}\mathcal{O}_n(x,y,z_1,z_2),
\end{equation}

\begin{equation}
\sum_{\substack{\{ \sigma \in B_{n}  :\, \sigma(n)<0 \\ \negg(\sigma) \equiv \epsilon \modue \}}}
(-1)^{\ell_B(\sigma)} x^{\emaj(\sigma)} y^{\edes(\sigma)}z_1^{\oneg(\sigma)}z_2^{\eneg(\sigma)} =-\frac{(-yz_1)^{\epsilon}z_1z_2p_{n+1}}{1-z_1z_2}\mathcal{O}_n(x,y,z_1,z_2).
\end{equation}
\end{thm}

As a consequence of Theorems \ref{fourcorners} and \ref{efourcorners} we can now easily obtain the generating function of
the odd flag-major index, the odd descent number, and the odd negative number, twisted by the alternating character of the hyperoctahedral group, and the corresponding even ones.
The following result is the odd and even analogue, for $y=z=1$, of \cite[Theorem 5.1]{AGR}.

\begin{cor}
\label{ell}
Let $n \geqslant 2$. Then
$$
\sum_{\sigma \in B_{n}} (-1)^{\ell_B (\sigma)} x^{\ofmaj(\sigma)} y^{\odes(\sigma)}  z^{\oneg(\sigma)} =
\left\lfloor  \frac{n}{2} \right\rfloor ! \, (1-xz)^{\lceil  \frac{n}{2} \rceil}
\prod_{i=1}^{\left\lfloor  \frac{n}{2} \right\rfloor} (1- y x^{2i}),
$$
and
$$\sum_{\sigma \in B_{n}} (-1)^{\ell_B (\sigma)} x^{\efmaj(\sigma)} y^{\edes(\sigma)}  z^{\eneg(\sigma)} = p_{n+1} \left\lfloor  \frac{n}{2} \right\rfloor!(1-xz)^{\left\lfloor  \frac{n}{2} \right\rfloor} \prod \limits_{i=0}^{\left\lfloor  \frac{n}{2} \right\rfloor} (1-yx^{2i}).
$$
\end{cor}

We note that Corollary \ref{ell} implies that $$\sum_{\sigma \in B_{n}} (-1)^{\ell_B (\sigma)} x^{\omaj(\sigma)} y^{\odes(\sigma)}=0$$
which is also implied by \cite[Theorem 3.2]{Reiner}.

 Note that $\sum_{\sigma \in B_{n}} (-1)^{\ell_B (\sigma)} x_{1}^{\omaj(\sigma)} x_{2}^{\emaj(\sigma)}$ does not factor
nicely, in general, for example, if $n=5$ then one obtains $(1-x_1)(1-x_1x_2)^2(1+x_2^2)(x_1^6x_2^4 - 2x_1^4x_2^2 + x_1^2x_2^4 + x_1^4 - 2x_1^2x_2^2 + 1)$.

As another simple consequence of Theorems \ref{fourcorners} and \ref{efourcorners}
we also obtain the generating function of $\oDmaj$, $\odes_D$ and $\oneg_D$, and the
corresponding even one for the only non-trivial one-dimensional character of the
even hyperoctahedral group. This result is the odd and even analogue, for $y=z=1$, of \cite[Theorem 4.8]{Biagioli-2006}.

\begin{cor}
\label{oDmaj}
Let $n \geqslant 2$. Then
$$
\sum_{\sigma \in D_{n}} (-1)^{\ell_D (\sigma)} x^{\oDmaj(\sigma)} y^{\odes_D(\sigma)} z^{\oneg_D(\sigma)}=
\left\lfloor  \frac{n}{2}  \right\rfloor !
(1-xz)^{\left\lfloor  \frac{n-1}{2} \right\rfloor}
\prod_{i=1}^{\left\lfloor  \frac{n}{2} \right\rfloor} (1-y x^{2i}),
$$
and
$$\sum_{\sigma \in D_{n}} (-1)^{\ell_D (\sigma)} x^{\eDmaj(\sigma)} y^{\edes_D(\sigma)}  z^{\eneg_D(\sigma)} =
                                                                                                            p_{n+1}\left\lfloor \frac{n}{2} \right \rfloor!(1+y)(1-xz)^{\left\lfloor  \frac{n-2}{2} \right\rfloor} \prod \limits_{i=1}^{\left\lfloor \frac{n}{2} \right \rfloor} (1-yx^{2i}).
$$
\end{cor}
\begin{proof}
For $\sigma \in D_n$ let $\tilde{\sigma} :=[\sigma(1), \ldots , \sigma(n-1), -\sigma(n)]$.
Then we have from our definitions, Proposition \ref{CombD} and Theorem \ref{fourcorners} that

\[
\begin{array}{ll}
\sum_{\sigma \in D_{n}} &(-1)^{\ell_D (\sigma)}x^{\oDmaj(\sigma)} y^{\odes_D(\sigma)} z^{\oneg_D(\sigma)}=\\
&=\displaystyle \sum_{\{ \sigma \in D_{n} : \sigma(n)>0 \}} (-1)^{\ell_B(\sigma)}
\sigma_o(x^2,y,zx,1)+\displaystyle \sum_{\{ \sigma \in D_{n} : \sigma(n)<0 \}} (-1)^{\ell_B(\tilde{\sigma})+1} \tilde{\sigma}_o(x^2,y,zx,1) \\
& = \displaystyle \left\lfloor \frac{n}{2} \right\rfloor !
(1-xz)^{\left\lfloor \frac{n-1}{2} \right\rfloor} \prod_{i=1}^{\left\lfloor \frac{n}{2} \right\rfloor} (1-y x^{2i}) - \displaystyle \sum_{\{ \tau \in B_n \setminus D_{n} : \tau(n)>0 \}} (-1)^{\ell_B(\tau)} \tau_o(x^2,y,zx,1) \\
& =\displaystyle \left\lfloor  \frac{n}{2} \right\rfloor !
(1-xz)^{\left\lfloor  \frac{n-1}{2} \right\rfloor}
\prod_{i=1}^{\left\lfloor  \frac{n}{2} \right\rfloor} (1-y x^{2i}).
\end{array}
\]
The second equation follows similarly from Theorem \ref{efourcorners}.
\end{proof}
We note that, as in the case of the hyperoctahedral group, $\sum \limits_{\sigma \in D_n}(-1)^{\ell_D(\sigma)}x_1^{\oDmaj(\sigma)}x_2^{\eDmaj(\sigma)}$ does not seem to factor, in general.

\subsection{The other characters}

We conclude by computing the generating function of the statistics
studied in this section twisted by the remaining one-dimensional characters of the hyperoctahedral
group.

As in the case of the alternating character, the following corollary can be deduced directly from
Theorems \ref{fourcorners} and \ref{efourcorners}. The next result gives the odd and even analogue, for $y=z=1$, of \cite[Theorem 6.1]{AGR}.

\begin{cor}
\label{ell+neg}
Let $n \geqslant 2$. Then
$$
\sum \limits_{ \sigma \in B_n} (-1)^{\ell_B(\sigma)+\negg(\sigma)}x^{\ofmaj(\sigma)} y^{\odes(\sigma)}z^{\oneg(\sigma)} = \left\lfloor  \frac{n}{2} \right\rfloor ! \, (1+xz)^{p_{n+1}}(1-xz)^{\lfloor  \frac{n}{2} \rfloor}
\prod \limits_{i=1}^{\left\lfloor  \frac{n}{2} \right\rfloor} (1- y x^{2i})
$$
and
$$\sum_{\sigma \in B_{n}} (-1)^{\ell_B (\sigma)+\negg(\sigma)} x^{\efmaj(\sigma)} y^{\edes(\sigma)}  z^{\eneg(\sigma)} =p_{n+1}\left\lfloor  \frac{n}{2} \right\rfloor!(1+y)(1-xz)^{\left\lfloor  \frac{n}{2} \right\rfloor} \prod \limits_{i=1}^{\left\lfloor  \frac{n}{2} \right\rfloor} (1-yx^{2i}).
$$
\end{cor}

One can check that $\sum_{\sigma \in B_{n}} (-1)^{\ell_B (\sigma)+\negg(\sigma)} x_{1}^{\omaj(\sigma)} x_{2}^{\emaj(\sigma)}$ does not factor explicitly in general.

To calculate the generating function of the statistics
studied in this section twisted by the remaining character
we begin with a reduction result.

For $\sigma \in B_n$ we let $|\sigma| := [|\sigma(1)|, \ldots, |\sigma(n)|]$.

\begin{pro} \label{prop neg}
  Let $n\geqslant 2$ and $S\subseteq [n]$. Then
$$
\sum \limits_{\substack{\{\sigma \in B_n:\\\Neg(\sigma)_o=S\}}}(-1)^{\negg(\sigma)}x^{\omaj(\sigma)}y^{\odes(\sigma)}=
\sum \limits_{\substack{\{\sigma \in B_n:\\ \Neg(\sigma)_o=S \\ \odes(|\sigma|)=0\}}}(-1)^{\negg(\sigma)}x^{\omaj(\sigma)}y^{\odes(\sigma)}
$$ and
$$\sum \limits_{\substack{\{\sigma \in B_n:\\\Neg(\sigma)_e=S\}}}(-1)^{\negg(\sigma)}x^{\emaj(\sigma)}y^{\edes(\sigma)}=
\sum \limits_{\substack{\{\sigma \in B_n: \\ \Neg(\sigma)_e=S \\ \edes(|\sigma|)=0\}}}(-1)^{\negg(\sigma)}x^{\emaj(\sigma)}y^{\edes(\sigma)}.$$
\end{pro}
\begin{proof}
Let $\sigma\in B_n$ and $i\in [n-1]$ be such that $i\equiv 1\modue$ and $|\sigma(i)|>|\sigma(i+1)|$. Let $\tilde{\sigma}:= [\sigma(1),...,\sigma(i),-\sigma(i+1),\sigma(i+2),...,\sigma(n)]$. Then $\negg(\tilde{\sigma})\equiv \negg(\sigma)+1 \modue$,
$\oneg(\tilde{\sigma})=\oneg(\sigma)$, $\odes(\tilde{\sigma})=\odes(\sigma)$ and hence $\omaj(\tilde{\sigma})=\omaj(\sigma)$.

The second formula is proved analogously.
\end{proof}

The next result completes the computation of the generating functions of the
statistics $\ofmaj$, $\odes$, and $\oneg$, and their even counterparts, twisted by the
one-dimensional characters of $B_n$. The result is the odd and even analogue, for $y=z=1$, of \cite[Theorem 6.2]{AGR}.
\begin{thm}
  Let $n\geqslant 2$. Then
$$\sum \limits_{\sigma \in B_n}(-1)^{\negg(\sigma)}x^{\ofmaj(\sigma)}y^{\odes(\sigma)}z^{\oneg(\sigma)}=\frac{n!}{2^{\left\lfloor \frac{n}{2} \right\rfloor}}
(1-xz)^{\left\lceil \frac{n}{2} \right\rceil} \prod\limits_{i=1}^{\left\lfloor \frac{n}{2} \right\rfloor}(1-yx^{2i}),$$
and
$$\sum \limits_{\sigma \in B_n}(-1)^{\negg(\sigma)}x^{\efmaj(\sigma)}y^{\edes(\sigma)}z^{\eneg(\sigma)}=\frac{n!}{2^{\left\lfloor \frac{n-1}{2} \right\rfloor}}
(1-xz)^{\left\lfloor \frac{n}{2} \right\rfloor} \prod\limits_{i=0}^{\left\lfloor \frac{n-1}{2} \right\rfloor}(1-yx^{2i}).
$$
\end{thm}
\begin{proof}
By Proposition \ref{prop neg} we have that

  \begin{eqnarray*}
  \sum \limits_{\sigma \in B_n}&(-1)^{\negg(\sigma)}&x^{\omaj(\sigma)}y^{\odes(\sigma)}z^{\oneg(\sigma)}=\\
    &=&
\sum\limits_{S\subseteq [n]_o} z^{|S|}\sum \limits_{\substack{\{\sigma \in B_n:\\\Neg(\sigma)_o=S\}}}(-1)^{\negg(\sigma)}x^{\omaj(\sigma)}y^{\odes(\sigma)} \\
     &=& \sum\limits_{S\subseteq [n]_o} z^{|S|}\sum \limits_{\substack{\{\sigma \in B_n:\\ \odes(|\sigma|)=0 \\ \Neg(\sigma)_o=S\}}}(-1)^{\negg(\sigma)}x^{\omaj(\sigma)}y^{\odes(\sigma)}.
\end{eqnarray*}
Now notice that, if $\sigma \in B_n$ is such that $\odes(|\sigma|)=0$ and $i \equiv 1 \modue$,
then $\sigma(i)>\sigma(i+1)$ if and only if $\sigma(i+1)<0$, so $D(\sigma)_o+1=\Neg(\sigma)_e$.
Therefore the previous sum equals
\begin{eqnarray*}
& & \sum\limits_{S\subseteq [n]_o} z^{|S|}\sum\limits_{T \subseteq [n]_e}  \sum \limits_{\substack{\{\sigma \in B_n:\\\Neg(\sigma)=S \uplus T\\ \odes(|\sigma|)=0\}}}(-1)^{|S|+|T|}x^{\frac{1}{2}\sum\limits_{t\in T}t}y^{|T|} \\
&=&|\{\sigma \in B_n: \odes(|\sigma|)=0, \Neg(\sigma)=S \cup T\}| \sum\limits_{S\subseteq [n]_o} (-z)^{|S|}\sum\limits_{T \subseteq [n]_e} (-y)^{|T|}x^{\frac{1}{2}\sum\limits_{t\in T}t}  \\
&=& \frac{n!}{2^{\left\lfloor \frac{n}{2} \right\rfloor}}(1-z)^{\left\lceil \frac{n}{2} \right\rceil} \prod\limits_{i=1}^{\left\lfloor \frac{n}{2} \right\rfloor}(1-yx^i).
 \end{eqnarray*}
The second equality follows analogously, using the reduction of Proposition \ref{prop neg}.

\end{proof}

The joint distribution of $\ofmaj$ and $\efmaj$ twisted by ``negative'' character does not seem to factor.

\section{Final comments}

It is clear that the most desirable property that one would like an ``odd major index'' to possess
is that it is equidistributed with the odd length.
It is easy to see that, if we require, as seems reasonable, such an odd major index to be
a function of the descent set, such an odd major index does not exist in general.
For example, one has that
\begin{eqnarray*}
\sum_{\pi \in S_5} \prod_{i \in D(\pi)} x_i & = & 1+4 x_4 + 9 x_3 + 6 x_3 x_4 +9x_2+16 x_2 x_4+11x_2 x_3
+4 x_2 x_3 x_4 + 4 x_1 \\ & & +11 x_1 x_4+ 16 x_1 x_3
 + 9 x_1 x_3 x_4 +6 x_1 x_2 +9 x_1 x_2 x_4 +
4 x_1 x_2 x_3 + x_1 x_2 x_3 x_4
\end{eqnarray*}
and one can check that there are no $j_1,j_2,j_3,j_4 \in \N$ such that
\[
\sum_{\pi \in S_5} \prod_{i \in D(\pi)} x^{j_i} = 1+12x+23x^2+48 x^3 + 23 x^4+ 12 x^5 +x^6
 =  \sum_{\pi \in S_5} x^{L(\pi)}
\]
where $L(\pi) := |\{ (i,j) \in [n]^2 : i<j, \pi(i) > \pi(j), i \not \equiv j \pmod{2}  \} |$
is the odd length of the symmetric group.

Similar computations show that no ``odd major index'' that depends only on the descent and negative sets exists in the hyperoctahedral groups
that is equidistributed with the odd length, where the odd length is the one defined in \cite{SV1}
and \cite{SV2}, and further studied in \cite{BC}, \cite{BC3}, and \cite{Lan}, namely
$$L_B(\sigma)=\frac{1}{2}|\{(i,j)\in [\pm n]^2 : i<j,\,\sigma(i)>\sigma(j),\,i\not\equiv j \modue\}|,$$ where $\sigma \in B_n$ and $\sigma(0):=0$.

\section{Acknowledgements}

The first author would like to thank Sylvie Corteel for useful and interesting conversations that led to the proof of Proposition \ref{unimod over}. This material is partly based upon work supported by the Swedish Research Council under
grant no. 2016-06596 while the first author was in residence at Institut Mittag-Leffler in Djursholm, Sweden during Spring 2020.

The first author is partially supported by the MIUR Excellence Department Project
CUP E83C18000100006.

\end{document}